\documentclass[11pt,leqno]{article}
\usepackage[margin=1in]{geometry} 
\usepackage{amssymb,amsfonts,amsmath,bbm,mathrsfs,stmaryrd,mathtools}
\usepackage{xcolor}
\usepackage{graphicx}
\usepackage{url}

\usepackage{extarrows}

\usepackage[shortlabels]{enumitem}
\usepackage{tensor}

\usepackage{xr}
\usepackage[T1]{fontenc}

\usepackage[colorlinks,
linkcolor=black!75!red,
citecolor=blue,
pdftitle={},
pdfproducer={pdfLaTeX},
pdfpagemode=None,
bookmarksopen=true,
bookmarksnumbered=true,
backref=page]{hyperref}

\usepackage{tikz}
\usetikzlibrary{arrows,calc,decorations.pathreplacing,decorations.markings,decorations.shapes,intersections,shapes.geometric,through,fit,shapes.symbols,positioning,decorations.pathmorphing}

\makeatletter
\DeclareRobustCommand\widecheck[1]{{\mathpalette\@widecheck{#1}}}
\def\@widecheck#1#2{%
    \setbox\z@\hbox{\m@th$#1#2$}%
    \setbox\tw@\hbox{\m@th$#1%
       \widehat{%
          \vrule\@width\z@\@height\ht\z@
          \vrule\@height\z@\@width\wd\z@}$}%
    \dp\tw@-\ht\z@
    \@tempdima\ht\z@ \advance\@tempdima2\ht\tw@ \divide\@tempdima\thr@@
    \setbox\tw@\hbox{%
       \raise\@tempdima\hbox{\scalebox{1}[-1]{\lower\@tempdima\box
\tw@}}}%
    {\ooalign{\box\tw@ \cr \box\z@}}}
\makeatother

\usepackage{braket}

\usepackage[amsmath,thmmarks,hyperref]{ntheorem}
\usepackage{cleveref}

\newcommand\nthalias[1]{\AddToHook{env/#1/begin}{\crefalias{lemma}{#1}}}

\nthalias{definition}
\nthalias{example}
\nthalias{examples}
\nthalias{remark}
\nthalias{remarks}
\nthalias{convention}
\nthalias{notation}
\nthalias{construction}
\nthalias{sketch}
\nthalias{theoremN}
\nthalias{propositionN}
\nthalias{corollaryN}
\nthalias{lemma}
\nthalias{proposition}
\nthalias{corollary}
\nthalias{theorem}
\nthalias{conjecture}
\nthalias{question}
\nthalias{assumption}
\nthalias{recollection}

\creflabelformat{enumi}{#2#1#3}

\crefname{section}{Section}{Sections}
\crefformat{section}{#2Section~#1#3} 
\Crefformat{section}{#2Section~#1#3} 

\crefname{subsection}{\S}{\S\S}
\AtBeginDocument{%
  \crefformat{subsection}{#2\S#1#3}%
  \Crefformat{subsection}{#2\S#1#3}%
}

\crefname{subsubsection}{\S}{\S\S}
\AtBeginDocument{%
  \crefformat{subsubsection}{#2\S#1#3}%
  \Crefformat{subsubsection}{#2\S#1#3}%
}

\theoremstyle{plain}

\newtheorem{lemma}{Lemma}[section]
\newtheorem{proposition}[lemma]{Proposition}

\newtheorem{theorem}[lemma]{Theorem}

\theoremstyle{plain}
\theoremnumbering{Alph}

\theoremstyle{plain}
\theorembodyfont{\upshape}
\theoremsymbol{\ensuremath{\blacklozenge}}

\newtheorem{definition}[lemma]{Definition}
\newtheorem{example}[lemma]{Example}

\newtheorem{remark}[lemma]{Remark}
\newtheorem{remarks}[lemma]{Remarks}

\crefname{definition}{definition}{definitions}
\crefformat{definition}{#2definition~#1#3} 
\Crefformat{definition}{#2Definition~#1#3} 

\crefname{ex}{example}{examples}
\crefformat{example}{#2example~#1#3} 
\Crefformat{example}{#2Example~#1#3} 

\crefname{exs}{example}{examples}
\crefformat{examples}{#2example~#1#3} 
\Crefformat{examples}{#2Example~#1#3} 

\crefname{remark}{remark}{remarks}
\crefformat{remark}{#2remark~#1#3} 
\Crefformat{remark}{#2Remark~#1#3} 

\crefname{remarks}{remark}{remarks}
\crefformat{remarks}{#2remark~#1#3} 
\Crefformat{remarks}{#2Remark~#1#3} 

\crefname{convention}{convention}{conventions}
\crefformat{convention}{#2convention~#1#3} 
\Crefformat{convention}{#2Convention~#1#3} 

\crefname{notation}{notation}{notations}
\crefformat{notation}{#2notation~#1#3} 
\Crefformat{notation}{#2Notation~#1#3} 

\crefname{table}{table}{tables}
\crefformat{table}{#2table~#1#3} 
\Crefformat{table}{#2Table~#1#3}

\crefname{lemma}{lemma}{lemmas}
\crefformat{lemma}{#2lemma~#1#3} 
\Crefformat{lemma}{#2Lemma~#1#3} 

\crefname{proposition}{proposition}{propositions}
\crefformat{proposition}{#2proposition~#1#3} 
\Crefformat{proposition}{#2Proposition~#1#3} 

\crefname{propositionN}{proposition}{propositions}
\crefformat{propositionN}{#2proposition~#1#3} 
\Crefformat{propositionN}{#2Proposition~#1#3} 

\crefname{corollary}{corollary}{corollaries}
\crefformat{corollary}{#2corollary~#1#3} 
\Crefformat{corollary}{#2Corollary~#1#3} 

\crefname{corollaryN}{corollary}{corollaries}
\crefformat{corollaryN}{#2corollary~#1#3} 
\Crefformat{corollaryN}{#2Corollary~#1#3} 

\crefname{theorem}{theorem}{theorems}
\crefformat{theorem}{#2theorem~#1#3} 
\Crefformat{theorem}{#2Theorem~#1#3} 

\crefname{theoremN}{theorem}{theorems}
\crefformat{theoremN}{#2theorem~#1#3} 
\Crefformat{theoremN}{#2Theorem~#1#3} 

\crefname{enumi}{}{}
\crefformat{enumi}{#2#1#3}
\Crefformat{enumi}{#2#1#3}

\crefname{assumption}{assumption}{Assumptions}
\crefformat{assumption}{#2assumption~#1#3} 
\Crefformat{assumption}{#2Assumption~#1#3} 

\crefname{construction}{construction}{Constructions}
\crefformat{construction}{#2construction~#1#3} 
\Crefformat{construction}{#2Construction~#1#3} 

\crefname{sketch}{sketch}{Sketches}
\crefformat{sketch}{#2sketch~#1#3} 
\Crefformat{sketch}{#2Sketch~#1#3} 

\crefname{recollection}{recollection}{Recollectiones}
\crefformat{recollection}{#2recollection~#1#3} 
\Crefformat{recollection}{#2Recollection~#1#3} 

\crefname{question}{question}{Questions}
\crefformat{question}{#2question~#1#3} 
\Crefformat{question}{#2Question~#1#3} 

\crefname{equation}{}{}
\crefformat{equation}{(#2#1#3)} 
\Crefformat{equation}{(#2#1#3)}

\numberwithin{equation}{section}

\theoremstyle{nonumberplain}
\theoremsymbol{\ensuremath{\blacksquare}}

\newtheorem{proof}{Proof}
\newcommand\pf[1]{\newtheorem{#1}{Proof of \Cref{#1}}}

\newcommand\bC{{\mathbb C}}

\newcommand\bG{{\mathbb G}}
\newcommand\bH{{\mathbb H}}

\newcommand\bP{{\mathbb P}}
\newcommand\bQ{{\mathbb Q}}
\newcommand\bR{{\mathbb R}}
\newcommand\bS{{\mathbb S}}
\newcommand\bT{{\mathbb T}}
\newcommand\bU{{\mathbb U}}

\newcommand\bZ{{\mathbb Z}}

\newcommand\cA{{\mathcal A}}
\newcommand\cB{{\mathcal B}}
\newcommand\cC{{\mathcal C}}

\newcommand\cE{{\mathcal E}}
\newcommand\cF{{\mathcal F}}
\newcommand\cG{{\mathcal G}}
\newcommand\cH{{\mathcal H}}

\newcommand\cM{{\mathcal M}}

\newcommand\cP{{\mathcal P}}

\newcommand\cX{{\mathcal X}}

\newcommand\ol{\overline}
\newcommand\wt{\widetilde}

\DeclareMathOperator{\id}{id}

\DeclareMathOperator{\End}{\mathrm{End}}

\DeclareMathOperator{\im}{\mathrm{im}}

\newcommand{\cat}[1]{\textsc{#1}}

\newcommand{\qedhere}{\mbox{}\hfill\ensuremath{\blacksquare}}

\newcommand{\xrightarrowdbl}[2][]{%
  \xrightarrow[#1]{#2}\mathrel{\mkern-14mu}\rightarrow
}

\title{Equivariant Banach-bundle germs}
\author{Alexandru Chirvasitu}

\begin{document}

\date{}

\newcommand{\Addresses}{{
  \bigskip
  \footnotesize

  \textsc{Department of Mathematics, University at Buffalo}
  \par\nopagebreak
  \textsc{Buffalo, NY 14260-2900, USA}  
  \par\nopagebreak
  \textit{E-mail address}: \texttt{achirvas@buffalo.edu}


}}

\maketitle

\begin{abstract}
  Consider a continuous bundle $\mathcal{E}\to X$ of Banach/Hilbert spaces or Banach/$C^*$-algebras over a paracompact base space, equivariant for a compact Lie group $\mathbb{U}$ operating on all structures involved. We prove that in all cases homogeneous equivariant subbundles extend equivariantly from $\mathbb{U}$-invariant closed subsets of $X$ to closed invariant neighborhoods thereof (provided the fibers are semisimple in the Banach-algebra variant). This extends a number of results in the literature (due to Fell for non-equivariant local extensibility around a single point for $C^*$-algebras and the author for semisimple Banach algebras). The proofs are based in part on auxiliary results on (a) the extensibility of equivariant compact-Lie-group principal bundles locally around invariant closed subsets of paracompact spaces, as a consequence of equivariant-bundle classifying spaces being absolute neighborhood extensors in the relevant setting and (b) an equivariant-bundle version of Johnson's approximability of almost-multiplicative maps from finite-dimensional semisimple Banach algebras with Banach morphisms. 
\end{abstract}


\vspace{.5cm}

\noindent{MSC 2020: 46L85; 55R91; 54D20; 46M20; 55R10; 55R15; 54C55; 55M15
  
}


\section*{Introduction}

\emph{Bundles} feature below in two related but distinct senses, to be distinguished by context:
\begin{itemize}[wide]
\item On the one hand there are the familiar (vector/fiber/principal)-bundle, always, here, \emph{locally trivial} \cite[Definition 4.7.2]{hus_fib}, background on which will be referenced variously from \cite{hjjm_bdle,hus_fib,ms-cc,steen_fib} and a number of other sources. Some matters of convenience:
  \begin{itemize}[wide]
  \item We freely (and sometimes tacitly) pass back and forth \cite[Assertion 18.3.4]{hjjm_bdle} between rank-$q$ vector or $q\times q$-matrix algebra bundles with their corresponding \emph{principal bundles} \cite[Definition 5.2.2]{hjjm_bdle} over $GL(q)$ or the projective general liner group $PGL(q):=GL(q)/\left(\text{central }\bC^{\times}\right)$ respectively.
    
  \item It is harmless by the selfsame \cite[Assertion 18.3.4]{hjjm_bdle} to furthermore work with \emph{Hermitian} vector bundles, i.e. \cite[\S 14.1]{ms-cc} those equipped with appropriately compatible Hilbert-space structures on their fibers. In principal-bundle language this has the effect of reducing the groups to $U(q)<GL(q)$ and $PU(q)<PGL(q)$ respectively. 
  \end{itemize}

\item On the other hand, there are the \emph{Banach bundles}
  \begin{equation*}
    \cA
    \xrightarrowdbl[\quad\text{continuous open}\quad]{\quad\pi\quad}
    X
    ,\quad
    \text{Banach-space fibers }\left(\cA_x:=\pi^{-1}x,\ \|\cdot\|_x\right)
  \end{equation*}
  of \cite[Definition II.13.4]{fd_bdl-1}, \cite[Definition 1.1]{dg_ban-bdl}, \cite[\S 3]{zbMATH03539905} and recalled below in \Cref{def:eqv.ban.bdl}. They \emph{can} be locally trivial in the sense of \cite[Remark II.13.9]{fd_bdl-1} but need not be. Throughout the paper they will for the most part be \emph{continuous} in the sense that the globally-defined norm
  \begin{equation}\label{eq:glob.nrm}
    \cA
    \xrightarrow{\quad\|\cdot\|\quad}
    \bR_{\ge 0}
    ,\quad
    \|\cdot\||_{\cA_x}
    :=
    \|\cdot\|_x
  \end{equation}
  is continuous (\emph{(F)} Banach bundles in \cite[post Definition 1.1, p.8]{dg_ban-bdl} terminology).

  In full generality (the \emph{(H)} Banach bundles of loc. cit.) \Cref{eq:glob.nrm} would be only \emph{upper semicontinuous} \cite[\S 1.3]{hk_shv-bdl}. Bundle base spaces and topological groups are always assumed Hausdorff absent explicit warnings.
\end{itemize}

Per standard conventions (e.g. \cite[Definition 2.2]{dupre_hilbund-1} or \cite[Definition IV.1.4.1]{blk} in a slightly different but related setting), Banach bundles are \emph{(sub-)homogeneous} if all fibers have the same finite dimension (or, respectively, there is an upper finite bound on fiber dimensions). For occasional added specificity, \emph{$n$-homogeneity} and \emph{$n$-sub-homogeneity} refer to the fibers being $n$ or $(\le n)$-dimensional respectively. 

By way of a recollection before stating some of the main results, we reprise the usual Banach-bundle definitions in a slightly expanded framework to afford \emph{equivariance} with respect to a group operating on the entire package, as this will be the framework of interest throughout the paper (cf. \cite[Definition 1.1]{dg_ban-bdl}).

\begin{definition}\label{def:eqv.ban.bdl}
  For a topological (usually compact) group $\bU$ a \emph{$\bU$-equivariant Banach bundle} over $\Bbbk\in \{\bC,\bR\}$ (shorthand\footnote{The term is intended to avoid the confusion that would likely be caused by \emph{$\bU$-bundle}: the latter might be understood as having $\bU$ for its \emph{structure group} \cite[Definition 4.5.1]{hus_fib}.}: \emph{Banach bundle$_{\bU}$}) consists of
  \begin{itemize}[wide]
  \item a continuous open $\bU$-equivariant map $\cE\xrightarrowdbl{\pi}X$ for topological spaces $\cE$ and $X$ equipped with continuous $\bU$-actions;

  \item with the \emph{fibers} $\cE_x:=\pi^{-1}x$, $x\in X$ carrying Banach-space structures compatible with the $\bU$-actions (which thus restrict to jointly-continuous representations $\bU\times \cE_x\to \cE_x$);

  \item so that scalar multiplication and addition are continuous;

  \item and the norm on $\cE$ is continuous (for \emph{continuous} or \emph{(F)} Banach bundles) or upper semicontinuous (for \emph{(H)} Banach bundles). 
  \end{itemize}
  Equivariant Hilbert-space or Banach-algebra or $C^*$-algebra bundles are defined similarly, with the fibers $\cE_x$ carrying those richer structures instead. 
\end{definition}

The common theme unifying the various threads of the discussion below is the extent to which singularities of Banach/Banach-algebra/$C^*$ bundles preclude the existence of better-behaved (e.g. locally-trivial) larger ambient bundles or, alternatively, subbundles. Several situations come to mind where understanding Banach-bundle (mis)behavior at and around fiber-dimension jumps is at the very least useful.
\begin{enumerate}[(a),wide]
\item\label{item:motiv.brnch} The \emph{non-commutative branched covers} introduced in \cite[Deﬁnition 1.2]{pt_brnch} are unital $C^*$ embeddings $B\le A$ which satisfying convenient finiteness conditions (being of \emph{finite index} \cite[Definition 2]{fk_fin-ind}) not recalled here explicitly. Specialized to commutative $B$, they will be subhomogeneous $C^*$ bundles studied in \cite{bg_cx-exp,2409.17807v1,2409.03531v2}.
  
  It is in that setting that the embeddability of a subhomogeneous $C^*$ bundle into a homogeneous one becomes relevant: homogeneity (hence also homogeneous embeddability) will ensure \cite[Proposition 3.4]{bg_cx-exp} the desired finiteness constraints hold, whereas they need not if the subhomogeneous bundle's fiber-dimension jump loci are sufficiently ill-behaved (per \cite[Theorem 1.3]{2409.17807v1} addressing \cite[Problem 3.11]{bg_cx-exp}, say). 

\item\label{item:motiv.sph} An adjacent theme to embeddability into homogeneous bundles is their \emph{prolongation} from closed subsets of the base space; this is what the title's \emph{germs} refer to: the term usually refers to objects (sheaf sections \cite[p.2]{bred_shf_2e_1997}, functions \cite[\S 27.18]{km_glob}) defined locally around a target (point, closed subset, etc.), identified if in agreement on a sufficiently small neighborhood of that target.

  One instance of the problem (which did in fact provide motivation for the present work) arises in relation to studying the \emph{non-commutative spheres} $\bS_{\theta}^{2n-1}$ introduced in \cite[Definition 2.1]{no_sph} as objects dual to the unital $C^*$-algebras generated universally by
  \begin{equation*}    
    \text{normal }T_k,\quad k\in \left\{1..n\right\}
    \quad:\quad
    \left[
    \begin{aligned}
      T_{\ell} T_k
      &=
        \exp\left(2\pi i \theta_{k\ell}\right) T_k T_{\ell},\quad\forall k,\ell\\
      \sum_{k=1}^n T_k^* T_k
      &=1
    \end{aligned}
    \right.
  \end{equation*}
  for skew-symmetric $\theta\in M_n(\bR)$. The latter is the same type of deformation parameter featuring in the definition of the \emph{non-commutative tori} $\bT^n_{\theta}$ (\cite[p.193]{rief_case}, \cite[\S 12.2]{gbvf_ncg}), the difference being that for these the generators $U_k$ would be unitary rather than normal.

  For rational deformation parameters $\theta\in M_n(\bQ)$ the $\bS^{2n-1}_{\theta}$ are non-commutative branched covers over classical spheres $\bS^{2n-1}$ in the sense of \Cref{item:motiv.brnch} (e.g. by a conjunction of \cite[Theorem B]{2508.04922v1} and \cite[Theorem A]{2409.03531v2}), particularly pleasant for $n=2$ \cite[Proposition 4.5]{cp_cont-equiv_xv3}:
  \begin{equation*}
    \cA\xrightarrowdbl{\quad \pi\quad}\bS^3
    ,\quad
    \cA_x\cong
    \begin{cases}
      M_q&x\not\in h^{-1}\left(p_{\pm}\right)\\
      \bC^q&\text{otherwise},
    \end{cases}
    ,\quad
    \begin{aligned}
      \bS^3
      &\xrightarrowdbl[\ \text{Hopf fibration}\ ]{\quad h\quad}\bS^2\\
      p_{\pm}&:=\text{antipodal pair}
    \end{aligned}    
  \end{equation*}
  where $q$ is the lowest-terms denominator of the rational $\theta_{12}$.

  It is useful, in computing the dimension invariants attached by \cite[Definitions 3.1 and 3.20]{hajacindex_xv2} (also \cite[Deﬁnition 2.3]{cpt_dim}) to finite-group actions on $\bS^{2n-1}_{\theta}$, to extend the $\bC^q$-fibered restriction of $\cA$ from one of the exceptional circles to a neighborhood thereof, equivariantly, as a piece of auxiliary scaffolding; \Cref{th:prlng.homog.abdls} affords this.
\end{enumerate}

To proceed to an outline of the paper's contents, the first extensibility result considered below reads as follows. 

\begin{theorem}\label{th:prlng.homog.bdls}
  \begin{enumerate}[(1),wide]
  \item Let $\bU$ be a compact Lie group, $\cE\xrightarrowdbl{\hspace{0pt}} X$ a continuous Banach bundle$_{\bU}$ over a paracompact base space, $Z\subseteq X$ a closed $\bU$-subset, and $\cF\le \cE|_Z$ a locally-trivial $n$-homogeneous subbundle$_{\bU}$.
    
    There are
    \begin{itemize}[wide]
    \item a closed $\bU$-invariant neighborhood $W\supseteq Z$ in $X$;
      
    \item and a locally-trivial $n$-homogeneous subbundle$_{\bU}$ $\wt{\cF}\xrightarrowdbl{\hspace{0pt}} W$ of $\cE|_W$;

    \item with $\wt{\cF}|_Z=\cF$. 
    \end{itemize}

  \item The analogous result holds also for continuous Hilbert bundles$_{\bU}$. 
  \end{enumerate}  
\end{theorem}

The multiplicative counterpart to \Cref{th:prlng.homog.bdls} is as follows. 

\begin{theorem}\label{th:prlng.homog.abdls}
  \begin{enumerate}[(1),wide]
  \item\label{item:th:prlng.homog.abdls:ban} Let $\bU$ be a compact Lie group, $\cA\xrightarrowdbl{\hspace{0pt}} X$ a continuous (unital) Banach-algebra bundle$_{\bU}$ over a paracompact base space, $Z\subseteq X$ a closed $\bU$-subset, and $\cB\le \cA|_Z$ a locally-trivial algebra subbundle$_{\bU}$ with finite-dimensional semisimple fibers.
    
    There are
    \begin{itemize}[wide]
    \item a closed $\bU$-invariant neighborhood $W\supseteq Z$ in $X$;
      
    \item and a locally-trivial subbundle$_{\bU}$ $\wt{\cB}\xrightarrowdbl{\hspace{0pt}} W$ of $\cA|_W$;

    \item with $\wt{\cB}|_Z=\cB$. 
    \end{itemize}

  \item\label{item:th:prlng.homog.abdls:cast} The analogous result holds also for continuous $C^*$ bundles$_{\bU}$. 
  \end{enumerate}  
\end{theorem}

\Cref{th:prlng.homog.abdls} generalizes \cite[Theorem 3.1]{fell_struct} (in the $C^*$ setting) and \cite[Proposition 2.8]{2409.03531v2} (for both $C^*$ and Banach-algebra bundles) by allowing equivariance and extensibility locally around arbitrary ($\bU$-invariant) closed subsets rather than single points. 

The proofs of \Cref{th:prlng.homog.bdls,th:prlng.homog.abdls} rely on a few technical asides presumably of some interest in their own right:
\begin{enumerate}[(a),wide]
\item\label{item:intro.equiv.bdls} There is, first, the ability to extend $\bU$-equivariant principal $\bG$-bundles locally around $\bU$-invariant closed subspaces $Z\subseteq X$ of paracompact $\bU$-spaces. This forms the object of \Cref{pr:princ.bdl.ext}, and is reminiscent of familiar results on he extensibility of (bundle or sheaf) \emph{sections} from closed subspaces: \cite[Theorem I.5.10]{kar_k_1978}, \cite[Th\'eor\`eme 3.3.1]{god_faisc_1958}, \cite[Proposition 1.1]{seg}, etc.

\item\label{item:intro.amnm} Secondly, in order to leverage the ``purely-linear'' \Cref{th:prlng.homog.bdls} into the ``multiplicative'' \Cref{th:prlng.homog.abdls}, we employ an equivariant bundle version of Johnson's principle \cite[Theorem 3.1, Corollary 3.2, Theorem 7.2]{john_approx} that almost-multiplicative linear ($*$-)maps from semisimple finite-dimensional Banach/$C^*$-algebras into arbitrary Banach/$C^*$-algebras are close to morphisms in the apposite category. This is one instance of what has since come to be termed \emph{Hyers-Ulam-Rassias stability} \cite[Introduction]{jung_hur-stab} and the family-wide, equivariant adaptation of Johnson's arguments can be regarded as fitting in that broad framework. 
\end{enumerate}

In reference to item \Cref{item:intro.equiv.bdls} above it becomes topical to contrast the various topologies one typically equips the total space of the \emph{(Milnor-model) universal bundle} $E\bG$ \cite[Definition 7.2.7]{hjjm_bdle} of a(n always Hausdorff) topological group. \Cref{se:colim.top} revolves around that issue:
\begin{itemize}[wide]
\item \Cref{th:disc.cone} observes that the principal action $E\bG\times \bG\to E\bG$ is discontinuous for the colimit topology on $E\bG$ whenever, say, $\bG\times -$ fails to preserve the topological quotient
  \begin{equation*}
    \bG\times [0,1]
    \xrightarrowdbl{\quad}
    \text{\emph{cone} }
    \cC \bG:=\bG\times [0,1]\big/\bG\times \{0\}.
  \end{equation*}
  Indeed, such groups fail to act continuously even on the truncated total space
  \begin{equation*}
    E_1\bG=\bG*\bG
    \underset{\text{closed}}{\subseteq}
    E\bG,
  \end{equation*}
  with `*' denoting the \emph{topological join} \cite[p.9]{hatch_at} recalled in \Cref{eq:x.ast.y}. As simple a group as $(\bQ,+)$ will exhibit this behavior: \Cref{pr:q.discont.act.join}.
  
\item while \Cref{th:lb.cc} notes that by contrast, a group with \emph{countably-compact} power $\bG^n$ always acts continuously on its colimit-topologized truncated universal space $E_n \bG=\bG^{*(n+1)}$. 
\end{itemize}

\subsection*{Acknowledgments}

I am grateful for illuminating comments and suggestions from B. Badzioch, M. Megrelishvili and M. Ramachandran.


\section{Banach-bundle extensibility and embeddability}\label{se:bdl.ext}

We refer the reader to some of the usual sources (cited as needed) for familiar point-set/algebraic-topology vocabulary (\emph{paracompact} spaces \cite[Definition 20.6]{wil_top}, \emph{CW-complexes} \cite[\S 8.3]{td_alg-top}, etc.). Algebras (and derivative notions: algebra morphisms, algebra bundles, etc.) are assumed unital barring explicit caution. We write
\begin{equation*}
  \Gamma_{\bullet}(\cE)
  \quad\text{or}\quad
  \Gamma_{\bullet}(\pi)
  ,\quad
  \bullet\in \{\text{blank},b,0\}
\end{equation*}
for arbitrary, bounded or, respectively, infinity-vanishing sections of a Banach bundle $\cE\xrightarrowdbl{\pi}X$ (the last only for locally compact $X$).


Throughout the paper,
\begin{equation}\label{eq:x.ast.y}
  X*Y
  :=
  X\times Y\times [0,1]
  \bigg/
  \left(
    \begin{aligned}
      (x,y,0)&\sim (x,y',0)\\
      (x,y,1)&\sim (x',y,1)\\
    \end{aligned}
  \right)
\end{equation}
denotes the \emph{join} \cite[\S 14.4.3]{td_alg-top} of two topological spaces. The quotient can be thought of as the space of convex combinations
\begin{equation*}
  t_X x+t_Y y
  ,\quad
  t_X,t_Y\in [0,1]
  ,\quad
  t_X+t_Y=1
  ,\quad
  x\in X
  ,\quad
  y\in Y,
\end{equation*}
equipped sometimes with the coarsest (\emph{coordinate}) topology making the functions
\begin{equation*}
  \begin{aligned}
    X*Y
    \ni
    t_X p_X+t_Y p_Y
    &\xmapsto{\quad c_{\bullet} \quad}
      t_{\bullet}\\
    c_{\bullet}^{-1}\left((0,1]\right)
    \ni
    t_X p_X+t_Y p_Y
    &\xmapsto{\quad\gamma_{\bullet}\quad}
      p_{\bullet}
  \end{aligned}      
  ,\quad
  \bullet\in \{X,Y\}
\end{equation*}
continuous and sometimes with the quotient topology resulting from \Cref{eq:x.ast.y}: cf. \cite[\S 14.4, Problem 10]{td_alg-top}. There is of course no discrepancy between the two when $X$ and $Y$ are compact Hausdorff, and the surrounding discussion will at any rate make it plain which version is in use. 

To amplify the Introduction's discussion,
recall two qualitatively distinct reasons why embedding a $C^*$ or Banach bundle into a locally-trivial one might not be possible.
\begin{itemize}[wide]
\item A topological obstruction of a bundle being locally trivial but in a sense ``not sufficiently so'' is illustrated in \cite[Example 3.6]{bg_cx-exp} (building on \cite[Example 3.5]{zbMATH05172034}). That discussion argues that a $C^*$ bundle $\cA$ over
  \begin{equation*}
    X:=\left(X_0:=\bigsqcup_{n\ge 1}\bC\bP^n\right) \sqcup\{\infty\}
    :=
    \text{\emph{one-point compactification} \cite[Problem 19A]{wil_top} of $X_0$},
  \end{equation*}
  locally trivial with fiber $M_2$ over $X_0$ and $\bC$ over $\infty$, fails to embed into a homogeneous $C^*$ bundle because $\cA|_{X_0}$ is not \emph{of finite type} \cite[Definition 3.5.7]{hus_fib}: trivialized by a finite open cover.

\item The $C^*$ bundle over $[-1,1]$ in \cite[Example 2.2]{2409.03531v2} on the other hand, trivial with fibers $\bC^2$ and $M_3$ respectively over $\{0\}$ and $[-1,1]\setminus \{0\}$ (so rather uninteresting topologically), is similarly homogeneously non-embeddable for algebraic reasons instead. We elaborate in \Cref{ex:not.same.trace}, expanding on the result just cited. 
\end{itemize}

\begin{example}\label{ex:not.same.trace}
  Consider closed embeddings
  \begin{equation*}
    \begin{tikzpicture}[>=stealth,auto,baseline=(current  bounding  box.center)]
      \path[anchor=base] 
      (0,0) node (l) {$X_0$}
      +(2,.5) node (u) {$X_-$}
      +(2,-.5) node (d) {$X_+$}
      +((4,0) node (al) {$A_0$}
      +(6,.5) node (au) {$A_-$}
      +(6,-.5) node (ad) {$A_+$}
      ;
      \draw[right hook->] (l) to[bend left=6] node[pos=.5,auto] {$\scriptstyle \iota_-$} (u);
      \draw[right hook->] (l) to[bend right=6] node[pos=.5,auto,swap] {$\scriptstyle \iota_+$} (d);
      \draw[right hook->] (al) to[bend left=6] node[pos=.5,auto] {$\scriptstyle \theta_-$} (au);
      \draw[right hook->] (al) to[bend right=6] node[pos=.5,auto,swap] {$\scriptstyle \theta_+$} (ad);
    \end{tikzpicture}
  \end{equation*}
  of compact Hausdorff spaces and finite-dimensional $C^*$-algebras respectively, and define a continuous $C^*$ bundle $\cA\xrightarrowdbl{\hspace{0pt}}X$ over the pushout $X:=X_-\sqcup_{X_0}X_+$ as that whose space $\Gamma(\cA)$ of sections is the pullback
  \begin{equation*}
    \begin{tikzpicture}[>=stealth,auto,baseline=(current  bounding  box.center)]
      \path[anchor=base] 
      (0,0) node (l) {$C_{\iota_-,\theta_-}(X_-,A_-)$}
      +(3,.5) node (u) {$\Gamma(\cA)$}
      +(3,-.5) node (d) {$C(X_0,A_0)$}
      +(6,0) node (r) {$C_{\iota_+,\theta_+}(X_+,A_+)$}
      ;
      \draw[->>] (u) to[bend right=6] node[pos=.5,auto] {$\scriptstyle $} (l);
      \draw[->>] (u) to[bend left=6] node[pos=.5,auto] {$\scriptstyle $} (r);
      \draw[->>] (l) to[bend right=6] node[pos=.5,auto,swap] {$\scriptstyle $} (d);
      \draw[->>] (r) to[bend left=6] node[pos=.5,auto,swap] {$\scriptstyle $} (d);
    \end{tikzpicture}
  \end{equation*}
  in the category of unital \emph{$C(X)$-algebras} \cite[Definition 1.5]{zbMATH04056334}, where
  \begin{equation*}
    C_{\iota_{\pm},\theta_{\pm}}(X_{\pm},A_{\pm})
    :=
    \left\{
      X_{\pm}
      \xrightarrow[\quad\text{continuous}\quad]{\quad f\quad}
      A_{\pm}
      \ :\
      \im f\circ \iota_{\pm}
      \subseteq
      \theta_{\pm}\left(A_0\right)
    \right\}
  \end{equation*}
  (a variant of the \emph{clutching construction} familiar in bundle theory: \cite[Definition 2.7.2]{hjjm_bdle}, \cite[Proposition 10.7.1]{hus_fib}, etc.; cf. also the projective-module gluing procedure of \cite[\S 2]{miln_k}). 

  $\cA$ has fiber $A_0$ over $X_0$ and respectively $A_{\pm}$ over $X_{\pm}$, and will not embed into any homogeneous locally-trivial $C^*$ bundle over $X$ provided
  \begin{equation*}
    X_0
    \cap
    \overline{\left(X_-\setminus X_0\right)}
    \cap
    \overline{\left(X_+\setminus X_0\right)}
    \ne
    \emptyset
  \end{equation*}
  and no tracial states on $A_{\pm}$ restrict to a common tracial state on $A_0$ along $\theta_{\pm}$. \cite[Example 2.2]{2409.03531v2} is what the present construction specializes to for
  \begin{itemize}[wide]
  \item $\iota_{\pm}$ being the 0-endpoint embeddings into $[0,1]$ and $[-1,0]$;
  \item and $\theta_{\pm}$ two embeddings $\bC^2\lhook\joinrel\to M_3$ picking out projections of different ranks. 
  \end{itemize}
  The reader can easily supply other like samples. 
\end{example}

\pf{th:prlng.homog.bdls}
\begin{th:prlng.homog.bdls}
  The Banach and Hilbert versions function mostly in tandem (the relevant bundle structure group is in both cases $\bG:=U(n)$), up to the one distinction noted at the very end of the proof.
  
  By \Cref{pr:princ.bdl.ext} below the locally-trivial rank-$n$ vector bundle$_{\bU}$ $\cE|_Z$ extends to one such, $\cG$ say, over some closed $\bU$-neighborhood $W\supseteq Z$; switching focus to $W$ in order to avoid symbol proliferation, there is no loss in assuming $W=X$ to begin with.

  Now, $\cF\le \cE|_Z$ constitutes an isometric embedding $\cG|_Z\cong \cE|_Z$ so in particular a section over $Z$ of the bundle
  \begin{equation*}
    \cE\otimes \cG^*
    \cong
    \cH om(\cG,\cE)
  \end{equation*}
  (having employed the usual vector-bundle operations of tensoring, duality, hom: \cite[\S 5.6]{hus_fib}, \cite[(f) post Corollary 3.5]{ms-cc}, etc.). That section extends (globally) over $X$:
  \begin{itemize}[wide]
  \item non-$\bU$-equivariantly in first instance, by \cite[p.15]{dg_ban-bdl};
    
  \item whereupon we can average against the Haar probability measure of $\bU$ to obtain a $\bU$-equivariant morphism $\cG\to \cE$ agreeing with the original embedding over $Z$ (much as in \cite[Proposition 1.1]{seg}, for instance). 
  \end{itemize}
  Such an extension will still be an embedding
  \begin{equation*}
    \cG|_W
    \lhook\joinrel\xrightarrow{\quad \iota\quad}
    \cE|_W
    ,\quad
    \text{closed neighborhood }W\supseteq Z,
  \end{equation*}
  providing the desired \emph{Banach} bundle. In the more rigid setup of Hilbert bundles, an additional step will replace $\iota$ with the isometric component $\iota_{\mathrm{iso}}$ in the \emph{polar decomposition} \cite[\S I.5.2.2]{blk} $\iota=\iota_{\mathrm{iso}}|\iota|$.
\end{th:prlng.homog.bdls}

\begin{remarks}\label{res:gen.triv}
  \begin{enumerate}[(1),wide]
  \item For $\bU=\{1\}$ and \emph{trivial} (rather than locally trivial) $\cF$ \Cref{th:prlng.homog.bdls} reflects the fact (e.g. \cite[Corollary 2.10]{laz_bb-sel} or \cite[p.15]{dg_ban-bdl}) that bounded sections of (even (H), rather than (F)) Banach bundles extend from closed subsets of paracompact base spaces:
    \begin{itemize}[wide]
    \item having selected $n$ sections witnessing triviality across all of $Z$;

    \item extend locally around $Z$ by the previously-cited \cite[Corollary 2.10]{laz_bb-sel};

    \item whereby the assumed continuity of (the norm of) $\cE$ will ensure that those extensions are still linearly independent locally around $Z$, thus generating the desired $n$-homogeneous subbundle $\wt{\cF}\le \cE|_W$. 
    \end{itemize}

  \item There is naturally no reason (even assuming $\bU=\{1\}$) to expect the extension claimed by \Cref{th:prlng.homog.bdls} to go through \emph{globally}: $X$ might support nothing but trivial vector bundles (e.g. it might be contractible), while $\cF\xrightarrowdbl{\hspace{0pt}} Z$ is non-trivial. A concrete example would be $X:=\bR^3$, $\cE$ its tangent bundle and $Z\subset X$ the unit 2-sphere with $\cF\xrightarrowdbl{\hspace{0pt}} Z$ \emph{its} tangent bundle.
  \end{enumerate}  
\end{remarks}

Given the sophisticated toolkit now available in the literature, \Cref{pr:princ.bdl.ext} below (used in the proof of \Cref{th:prlng.homog.bdls}) is a relatively simple remark. As it appears to be somewhat difficult to trace down in easily-referenced form; it is recorded here for completeness and convenience. 


To unwind the statement, we remind the reader of the rich theory of $\bU$-equivariant principal bundles developed and studied in various degrees of generality in sources such as \cite{MR3704236,MR650393,MR885537,MR711050,MR730709,MR3331607,MR1231040}. The main reference for the present portion of the discussion will be \cite[\S I.8]{td_transf-gp}. A brief recollection follows. 

\begin{definition}\label{def:eqvr.bdl}
  Let $\bG$ and $\bU$ be topological groups, with $\alpha:\bU\circlearrowright \bG$ a continuous left action of the latter on the former by topological-group automorphisms.

  \begin{enumerate}[(1),wide]
  \item A \emph{principal $\bG$-bundle$_{\alpha}$} $\cP\xrightarrowdbl{\pi} X$ (or \emph{$\bG$-bundle$_{\bU}$} when $\alpha$ is understood) is a locally trivial principal $\bG$-bundle over $X$ in the usual sense \cite[\S 14.1]{td_alg-top}, internal to the category of $\bU$-spaces in the sense that
    \begin{itemize}[wide]
    \item $B$ carries a continuous $\bU$-action;
    \item and $\pi$ is $\bU$-equivariant;
    \item along with the right $\bG$-action on $\cP$:
      \begin{equation*}
        \gamma\left(pg\right)=\left(\gamma p\right)\alpha_{\gamma}(g)
        ,\quad
        \forall\left(\gamma\in \bU,\ g\in \bG,\ p\in \cP\right).
      \end{equation*}
    \end{itemize}

  \item A \emph{local object$_{\alpha}$} is a principal $\bG$-bundle$_{\alpha}$ over a coset space $\bU/H$ for closed subgroups $H\le \bU$.

  \item A principal $\bG$-bundle$_{\alpha}$ is \emph{locally trivial$_{\alpha}$} (as opposed to just plain locally trivial as a $\bG$-bundle, which we always assume) if $B$ can be covered by open subsets $U\subseteq B$ with each $\pi^{-1}U\xrightarrow{\pi} U$ admitting a $\bG$-bundle$_{\alpha}$ morphism to a local object$_{\alpha}$.

    One is mostly interested in \emph{numerable} coverings $X=\bigcup_j U_j$: those for which there are locally finite \emph{partition of unity} \cite[\S 13.1]{td_alg-top} $\varphi_j$ with
    \begin{equation*}
      \mathrm{supp}\;\varphi_j
      :=
      \overline{\varphi_j^{-1}\bR_{>0}}
      \subseteq U_j
      ,\quad
      \forall j,
    \end{equation*}
    \emph{local finiteness} \cite[Definition 20.2]{wil_top} for a family of subsets meaning that every point has a neighborhood encountering only finitely many. 
    
    As we work virtually exclusively with paracompact spaces, the property is automatic for arbitrary open covers \cite[Theorem 13.1.3]{td_alg-top}.

  \item Given, in addition to $\alpha$ and a principal $\bG$-bundle$_{\alpha}$ $\cP\xrightarrowdbl{\hspace{0pt}}X$, a left $\bG$-space $F$ internal to the category of $\bU$-spaces, the usual \cite[Definition 5.3.1]{hjjm_bdle} \emph{associated-bundle} construction goes through $\bU$-equivariantly to yield a fiber bundle
    \begin{equation*}
      \cP[F]:=\cP\times_{\bG}F
      \xrightarrowdbl{\quad}
      X
    \end{equation*}
    with fiber $F$ with $\bU$ operating on everything in sight.
  \end{enumerate}
  For compact Lie $\bG$ and $\bU$, assumptions universally in place below, locally trivial principal $\bG$-bundles$_{\alpha}$ over paracompact spaces are automatically \cite[Theorem I.8.10]{td_transf-gp} locally trivial$_{\alpha}$. 
\end{definition}

\begin{proposition}\label{pr:princ.bdl.ext}
  Consider
  \begin{itemize}[wide]
  \item a compact Lie group $\bU$ acting via $\alpha:\bU\circlearrowright \bG$ on another;

  \item a closed $\bU$-equivariant embedding $Z\subseteq X$ into a paracompact $\bU$-space;

  \item and a principal $\bG$-bundle$_{\alpha}$ $\cP\xrightarrowdbl{\pi}Z$.
  \end{itemize}
  For some closed $\bU$-neighborhood $W\supseteq Z$ there is a principal $\bG$-bundle$_{\alpha}$ $\wt{\cP}\xrightarrowdbl{\hspace{0pt}}W$ with $\wt{\cP}|_Z\cong \cP$. 
\end{proposition}
\begin{proof}
  A brief review is in order of the \emph{universal} principal $\bG$-bundle$_{\alpha}$ of \cite[Theorem I.8.12]{td_transf-gp}, with slight alterations. It will be denoted here by $E_{\alpha}\bG\xrightarrowdbl{\hspace{0pt}}B_{\alpha}\bG$ by analogy to Milnor's (non-equivariant) model \cite[\S 14.4.3]{td_alg-top} $E\bG\xrightarrowdbl{\hspace{0pt}} B\bG$. 

  Defining
  \begin{equation*}
    E_{\alpha 0}\bG
    :=
    \text{disjoint union}
    \bigsqcup_{r\in R}
    \cP_r
  \end{equation*}
  with $\cP_r$ denoting (the total spaces of) representatives $\cP_r\xrightarrowdbl{} X_r$ for all isomorphism classes of local objects$_{\alpha}$, set
  \begin{equation}\label{eq:eens}
    E_{\alpha n}\bG
    :=
    \left(E_{\alpha 0}\bG\right)^{*(n+1)}
    ,\quad
    E_{\alpha}\bG
    :=
    \bigcup_n
    E_{\alpha n}\bG
    ,\quad
    B_{\alpha \bullet}\bG
    :=
    E_{\alpha \bullet}\bG/\bG
    \quad
    \left(\bullet\in \{n,\ \text{blank}\}\right)
  \end{equation}
  (so that $B_{\alpha} \bG$ is the \emph{classifying space} for numerable principal $\bG$-bundles$_{\alpha}$). While the matter is not absolutely crucial here, it seems appropriate to topologize the various spaces somewhat differently as compared to \cite[\S I.8]{td_transf-gp} (and \cite[\S 14.4.3]{td_alg-top}):
  \begin{itemize}[wide]
  \item the truncated joins $E_{\alpha n}\bG$ are quotient- rather than coordinate-topologized;
    
  \item while $E_{\alpha}\bG$ will be their topological \emph{colimit} in the usual \cite[Definition 2.6.6]{brcx_hndbk-1} category-theoretic sense. 
  \end{itemize}
  The compactness of $\bG$ and $\bU$ will ensure (as observed in \cite[footnote 1]{atiyahsegal} in the parallel, non-equivariant setting) that the $\bG$-action on $E_{\alpha\bG}$ is still continuous (in this regard, see also \Cref{re:not.top.compl}), and the universality proof goes through as in \cite[Theorem I.8.12]{td_transf-gp}. 

  Note next that we also have a topological-colimit realization
  \begin{equation*}
    \begin{aligned}
      \left(E_{\alpha}\bG\xrightarrowdbl{}B_{\alpha}\bG\right)
      &\cong
        \varinjlim_{\substack{\text{finite $F\subseteq R$}\\n\in \bZ_{\ge 0}}}
      \left(E_{\alpha\mid F,n}\bG\xrightarrowdbl{}B_{\alpha\mid F,n}\bG\right)\\
      E_{\alpha\mid F,n}
      :=
      \left(
      \bigsqcup_{r\in F}
      \cP_r
      \right)^{*(n+1)}
      &,\quad
        B_{\alpha\mid F,n}:=E_{\alpha\mid F,n}/\bG,
    \end{aligned}    
  \end{equation*}
  with $\left(E_{\alpha\mid F,n}\right)_{F,n}$ and $\left(B_{\alpha\mid F,n}\right)_{F,n}$ both non-decreasing in $(F,n)$ for the ordering
  \begin{equation*}
    (F,n)\preceq (F',n')
    \iff
    \left(
      F\subseteq F'
      \wedge
      n\le n'
    \right).
  \end{equation*}
  $\bU$ and $\bG$ both being compact Lie, the individual $E_{F,n}$ and $B_{F,n}$ are not difficult to give finite \emph{$\bU$-CW-complex} structures (\cite[\S II.1]{td_transf-gp} calls these \emph{$\bU$-equivariant CW-complexes}).

  Recall (e.g. \cite[p.292]{zbMATH02136472}) that a \emph{$\bU$-absolute neighborhood extensor} for a class $\cX$ of $\bU$-spaces (or a \emph{$\bU$-\cat{ANE}($\cX$)}) is a $\bU$-space $Y$ with the property that $\bU$-equivariant maps
  \begin{equation*}
    Z\xrightarrow{\quad}Y
    ,\quad
    Z=\overline{Z}\subseteq X\in \cX
  \end{equation*}
  extend across neighborhoods of $Z$. Finite CW-complexes are well-known to be (non-equivariant) \cat{ANE}s for paracompact spaces \cite[Theorem II.17.2]{hu_rtrct} as an application of Hanner's \cite[Theorem II.17.1]{hu_rtrct} (=\cite[Theorem 19.2]{zbMATH03076371}), to the effect that spaces that are \emph{local} \cat{ANE}s for paracompact spaces are globally so. The latter local-to-global result persists \cite[Corollary 5.7]{zbMATH02136472} for paracompact $\bU$-spaces with compact $\bU$, and in like fashion one deduces that finite $\bU$-CW-complexes are $\bU$-\cat{ANE}s for paracompact spaces.
  
  We thus have a \emph{filtered} \cite[\S IX.1]{mcl_2e} union
  \begin{equation*}
    B_{\alpha}\bG
    =
    \bigcup_{F,n}
    B_{\alpha\mid F,n}\bG
  \end{equation*}
  of $\bU$-\cat{ANE}s for paracompact $\bU$-spaces. To conclude that bundle$_{\alpha}$-classifying maps extend from closed $\bU$-invariant $Z\subseteq X$ to neighborhoods $W\supseteq Z$ thereof it now suffices to apply \cite[Theorem 5.1]{zbMATH02136472}, noting that a classifying map satisfies that result's hypothesis: every point $a\in Z$ has a neighborhood mapped into an individual $B_{\alpha\mid F,n}\bG$.
\end{proof}

We next turn to the algebra-bundle variant of \Cref{th:prlng.homog.bdls}. 

\pf{th:prlng.homog.abdls}
\begin{th:prlng.homog.abdls}

  
  There is also no harm in assuming the fibers of $\cB$ all isomorphic to a fixed (finite-dimensional, semisimple Banach or $C^*$-)algebra $B$, for their isomorphism classes will in any partition $Z$ into clopen sets.
  
  \begin{enumerate}[label={},wide]
  \item\textbf{\Cref{item:th:prlng.homog.abdls:ban}} In first instance, \Cref{pr:princ.bdl.ext} applied the automorphism group $\bG$ of some finite-dimensional $C^*$-algebra, i.e. a semidirect product of the form
    \begin{equation*}
      \left(\prod_{i=1}^k PU(n_i)\right)\rtimes \left(\text{finite group}\right),
    \end{equation*}
    will extend $\cB$ over a closed neighborhood (so we may as well assume globally) to a locally-trivial, $B$-fibered bundle$_{\bU}$ $\cM\xrightarrowdbl{\hspace{0pt}}X$ with $\cM|_A\cong \cB$. Precisely as in the proof of \Cref{th:prlng.homog.bdls}, we have a $\bU$-equivariant embedding
    \begin{equation*}
      \cM|_{W}
      \lhook\joinrel\xrightarrow{\quad\iota\quad}
      \cA|_W
      ,\quad
      \text{some closed neighborhood }W\supseteq Z
    \end{equation*}
    of \emph{vector} bundles, (unital and) multiplicative over $Z$. As a matter of convenience, a slight perturbation
    \begin{equation*}
      \iota
      \xmapsto{\quad}
      \bigg(
      \cM_x\ni a
      \xmapsto{\quad}
      \iota(a)+1_{\cA_x}-\iota\left(1_{\cM_x}\right)
      \in \cA_x
      \bigg)
      ,\quad
      x\in W
    \end{equation*}
    over sufficiently small $W$ followed by averaging over $\bU$ will further ensure $\iota$ is unital throughout $W$ (observe that this process will not have altered $\iota$ over the original smaller subspace $Z\subseteq X$). 
    
    $\iota$ pulls back a Banach-bundle norm on $\cM|_W$, though perhaps (away from $Z\subseteq W$) not quite what one usually calls an \emph{algebra} norm \cite[Definition 2.1.1]{dales_autocont}: it need not be sub-multiplicative in the sense that $\|ab\|\le \|a\|\cdot \|b\|$, nor is $\|1\|$ necessarily 1. Shrinking $W$ if necessary though, we can assume multiplication is uniformly bounded as a bilinear map \cite[post Theorem A.3.35]{dales_autocont}, and uniformly bounded on units:
    \begin{equation}\label{eq:unif.bds.almost.norm}
      \exists\left(K_2,\ K_0>1\right)
      \forall\left(x\in W,\ a,b\in \cM_x\right)
      \bigg(
      \|ab\|\le K_2\|a\|\cdot \|b\|
      \ \wedge\ 
      K_0^{-1}\le\|1_x\|\le K_0
      \bigg).
    \end{equation}
    All norms involved can also always be assumed $\bU$-invariant, by Haar-averaging. To summarize the current framework:
    \begin{itemize}[wide]
    \item we have the equivariant, isometric, unital embedding $\iota$ satisfying \Cref{eq:unif.bds.almost.norm} over $W\supseteq Z$;

    \item multiplicative over $Z$;

    \item and wish to substitute for it a (still equivariant, unital, and) multiplicative $\wt{\iota}$ over a possibly-smaller neighborhood of $Z$, with $\wt{\iota}|_Z=\iota$. 
    \end{itemize}
    This is where the assumed semisimplicity of the fibers $\cM_x$ will play a role: \cite[Corollary 3.2 via Theorem 3.1]{john_approx} proves that almost-multiplicative linear maps from semisimple finite-dimensional Banach algebras into arbitrary Banach algebras are closely approximable by multiplicative maps, and we need a version of this functioning
    \begin{itemize}[wide]
    \item in families (i.e. for algebra bundles rather than individual algebras);

    \item and $\bU$-equivariantly. 
    \end{itemize}    
    A brief reminder of how the proof of \cite[Theorem 3.1]{john_approx} functions will help adapt that proof to the present equivariant-bundle setup. The argument proceeds by a Newton-type approximation procedure that, in current notation, would be
    \begin{equation*}
      \iota_0:=\iota
      ,\quad
      \forall\left(n\ge 0\right)
      \quad:\quad
      \iota_{n+1}
      :=
      \tau\iota_n      
    \end{equation*}
    where
    \begin{equation}\label{eq:vee.tau}
      \forall\left(
        \cM|_W
        \xrightarrow[\quad]{\quad\varphi\quad}
        \cA|_W
      \right)
      \forall
      \left(
        s,t\in \Gamma_b\left(\cM|_W\right)
      \right)
      \quad:\quad
      \left[
        \begin{aligned}
          \tau\varphi
          &:=
            \varphi+\varphi(e_1)\varphi^{\vee}(e_2,-)\\
          \varphi^{\vee}(s,t)
          &:=
            \varphi(st)-\varphi(s)\varphi(t).
        \end{aligned}
      \right.
    \end{equation}
    and
    \begin{equation}\label{eq:si.e}
      e
      :=
      e_1\otimes e_2
      \in \Gamma_{b}\left(\cM|_W^{\otimes 2}\right)
    \end{equation}
    in \emph{Sweedler notation} \cite[Notation 1.4.2]{mont} (so there is a suppressed notation; the section does not restrict, necessarily, to simple tensors in individual fibers) restricts, at each individual $x\in W$, to a \emph{separability idempotent} \cite[Theorem 3]{cmz_frob-sep_2002}
    \begin{equation*}
      e_x=e_{1x}\otimes e_{2x}
      \in \cM_x^{\otimes 2}
      ,\quad
      \begin{aligned}
        m e_x
        &= e_xm,\ \forall m\in \cM_x\\
        e_{1x}e_{2x}
        &:=
          \text{multiplication}(e_x)=1
      \end{aligned}
    \end{equation*}
    of the finite-dimensional semisimple algebra $\cM_x$ (what \cite[Definition 1.9.19]{dales_autocont} calls a \emph{diagonal}). The aforementioned proof will function both family-wise and equivariantly, substituting for $\iota$ the limit
    \begin{equation*}
      \wt{\iota}:=\varinjlim_n \iota_n
      \quad
      \left(\text{unital, multiplicative and $\bU$-equivariant}\right)
    \end{equation*}
    provided
    \begin{enumerate}[(a),wide]
    \item\label{item:sml.w} we restrict attention to a shrunken $W$ over which $\iota^{\vee}$ is uniformly small (always achievable, given that $\iota^{\vee}\equiv 0$ over $Z$ itself);

    \item\label{item:si.sect} and separability idempotents for individual $\cM_x$, which do indeed always exist \cite[\S 10.7, Corollary b]{pierce_assoc} for semisimple finite-dimensional algebras over perfect fields, can be chosen coherently and equivariantly over $W$ (i.e. that there is indeed a $\bU$-invariant separability-idempotent \emph{section} $e$ as discussed). 
    \end{enumerate}
    The argument thus amounts to verifying \Cref{item:si.sect}, relegated to \Cref{pr:si.sect}. 
    
  \item\textbf{\Cref{item:th:prlng.homog.abdls:cast}} The additional ingredient, as compared to part \Cref{item:th:prlng.homog.abdls:ban}, is the $*$ structure. The earlier argument will carry through in the modified form employed in \cite[Theorem 7.2]{john_approx}, also tailored to $*$ structures.
    
    Define $*$ operators on one- and two-variable maps of \emph{$*$-vector spaces} \cite[Definition 3.1]{zbMATH05592601} (thence extending the definition in the obvious fashion for maps of $*$-structured bundles) by
    \begin{equation*}
      f^*(a)
      :=
      f\left(a^*\right)^*
      \quad\text{and}\quad
      f^*(a,b)
      :=
      f\left(b^*,a^*\right)^*
    \end{equation*}
    respectively. Observe next that the $*$ and $\vee$ operators commute: for $\iota$ as in the proof of part \Cref{item:th:prlng.homog.abdls:ban} we have
    \begin{equation*}
      \left(\iota^*\right)^{\vee}
      =
      \left(\iota^{\vee}\right)^{*}.
    \end{equation*}
    The iterative construction is now $\iota_n\rightsquigarrow \iota_{n+1}:=\tau_{sa} \iota_n$ (subscript suggesting \emph{self-adjoint}) with
    \begin{equation*}
      \tau_{sa}
      :=
      \frac{\tau+\tau^*}2
      \quad
      \left(\text{$\tau$ as in \Cref{eq:vee.tau}}\right).
    \end{equation*}
    The limit will be a $*$-morphism provided $\iota$ was one (which we are free to assume), and the separability-idempotent section \Cref{eq:si.e} satisfies
    \begin{equation*}
      \bigg(
      e^*=\left(e_1\otimes e_2\right)^*=e_1^*\otimes e_2^*
      \bigg)
      =
      \bigg(
      e_2\otimes e_1:=\text{flip map}(e)
      \bigg).
    \end{equation*}
    The same auxiliary \Cref{pr:si.sect} ensures this as well, concluding the proof.  \qedhere
  \end{enumerate}
\end{th:prlng.homog.abdls}

\begin{proposition}\label{pr:si.sect}
  Let $\bU$ be a compact Lie group and $\cA\xrightarrowdbl{\pi}X$ a real or complex locally-trivial algebra bundle$_{\bU}$ with finite-dimensional semisimple fibers over a paracompact base space.

  \begin{enumerate}[(1),wide]
  \item\label{item:pr:si.sect:si.plain} There is a $\bU$-invariant section of $e\in \Gamma\left(\cA^{\otimes 2}\right)$ restricting to a separability idempotent in each fiber.
    
  \item\label{item:pr:si.sect:si.ast} If $\cA$ is furthermore a bundle of $*$-algebras then $e$ as above can be chosen so that $e^*=\sigma(e)$ for the tensor-flip map $\sigma$.
  \end{enumerate}
\end{proposition}
\begin{proof}
  As noted in the course of the proof of \cite[Theorem 7.2]{john_approx}, $\frac{e+\sigma(e)^*}{2}$ is a separability idempotent whenever $e$ is. The procedure plainly also preserves $\bU$-invariance, so \Cref{item:pr:si.sect:si.ast} follows from \Cref{item:pr:si.sect:si.plain} (on which we focus for the duration of the proof). 

  The claim will follow once we verify that a finite-dimensional semisimple real or complex algebra $A$ is equipped with a canonical separability idempotent $e_A\in A^{\otimes 2}$, invariant under all $A$-automorphisms. It will be clear, once the construction has been described, that the individual $e_{\cA_x}$ vary continuously with $x\in X$ and thus constitute a(n automatically $\bU$-invariant) section.

  The left multiplication self-action $A\times A\to A$ induces a trace
  \begin{equation*}
    A
    \xrightarrow{\quad \mathrm{Tr}\quad}
    \Bbbk\in \left\{\bR,\bC\right\},
  \end{equation*}
  with the bilinear form $\mathrm{Tr}(-\cdot -)$ on $A$ non-degenerate by semisimplicity (thus making $A$ into a \emph{Frobenius algebra} \cite[Definition 3 and Theorem 4]{cmz_frob-sep_2002}). This gives an identification
  \begin{equation*}
    A
    \ni
    a
    \xmapsto[\quad\cong\quad]{\quad \varphi\quad}
    \mathrm{Tr}(-\cdot a)
    \in
    A^*
    \quad
    \left(\text{as left $A$-modules}\right),
  \end{equation*}
  hence also elements
  \begin{equation*}
    \begin{tikzpicture}[>=stealth,auto,baseline=(current  bounding  box.center)]
      \path[anchor=base] 
      (0,0) node (end) {$\End(A)$}
      +(3,0) node (aa*) {$A\otimes A^*$}
      +(6,0) node (aa) {$A\otimes A$}
      +(0,-1) node (1) {$\id_A$}
      +(3,-1) node (ff*) {$\sum_i f_i\otimes f_i^*$}
      +(6,-1) node (e) {$e$}
      +(0,-.5) node () {\rotatebox{90}{$\in$}}
      +(3,-.5) node () {\rotatebox{90}{$\in$}}
      +(6,-.5) node () {\rotatebox{90}{$\in$}}
      ;
      \draw[->] (end) to[bend left=6] node[pos=.5,auto] {$\scriptstyle $} (aa*);
      \draw[->] (aa*) to[bend left=6] node[pos=.5,auto] {$\scriptstyle \id\otimes \varphi^{-1}$} (aa);
      \draw[|->] (1) to[bend right=6] node[pos=.5,auto] {$\scriptstyle $} (ff*);
      \draw[|->] (ff*) to[bend right=6] node[pos=.5,auto] {$\scriptstyle $} (e);
    \end{tikzpicture}
  \end{equation*}
  where $(f_i)\subset A$ is any basis with $\left(f_i^*\right)\subset A^*$ the corresponding dual basis. The construction has thus far been automorphism-invariant as desired, and direct examination on a single complex matrix algebra (to which case we can always reduce by complexifying
  \begin{equation*}
    A
    \cong
    \prod_{i=1}^r M_{n_i}(\bR)
    \times
    \prod_{j=1}^c M_{n_j}(\bC)
    \times
    \prod_{k=1}^q M_{n_q}(\bH:=\text{quaternions})
    \quad
    \left(\text{\cite[Theorems 3.5 and 13.12]{lam_1st_2e_2001}}\right)
  \end{equation*}
  and passing to an individual factor) shows that $e$ is a separability idempotent. Indeed, running through the argument for $M_n(\bC)$ will produce $e=\sum_{i,j}e_{ij}\otimes e_{ji}$ with $e_{ij}$, $1\le i,j\le n$ denoting the usual matrix units. 
\end{proof}


\section{On and around colimit topologies on universal bundles}\label{se:colim.top}

Specializing \Cref{pr:princ.bdl.ext} back to trivial $\bU$, the total space $E\bG=E_{\alpha}\bG$ featuring in that proof for a (here, compact) topological group $\bG$ is variously equipped with two topologies in the literature:
\begin{enumerate}[(a),wide]
\item\label{item:stop} the colimit topology $\tau_{\varinjlim}$ induced by
  \begin{equation}\label{eq:egeng}
    E\bG
    =
    \bigcup_n E_n\bG
    ,\quad
    E_n\bG
    :=
    \bG^{*(n+1)},
  \end{equation}
  as in \cite[Definition 2.2.7]{hjjm_bdle}, \cite[\S 5]{miln_univ-2} or \cite[\S 2]{atiyahsegal} (also the CW topology featuring in the proof of \Cref{pr:princ.bdl.ext}), with $\bG^{*(n+1)}$ themselves quotient-topologized via
  \begin{equation*}
    \bG^{n+1}\times \left(\text{$n$-simplex }\Delta^n\right)
    \xrightarrowdbl{\quad}
    \bG^{*(n+1)}
    \quad
    \left(
      \text{e.g. \cite[Definition 7.2.3]{hjjm_bdle}}
    \right)
    ;
  \end{equation*}
\item\label{item:wtop} or the weakest topology $\tau_{w}$ rendering the maps
  \begin{equation*}
    \begin{aligned}
      E\bG
      \ni
      \sum_{n}t_n g_n
      &\xmapsto{\quad c_m\quad}
        t_m
        \in [0,1]\\
      c_m^{-1}((0,1])
      \ni
      \sum_{n}t_n g_n
      &\xmapsto{\quad \gamma_m\quad}
        g_m\in \bG
    \end{aligned}      
  \end{equation*}
  continuous (as in \cite[\S 3]{miln_univ-2}, \cite[\S 14.4.3]{td_alg-top} or \cite[\S 4.11]{hus_fib}). 
\end{enumerate}
\cite[footnote 1, p.3]{atiyahsegal} comments briefly on why they function equally well for compact $\bG$ for the purposes of bundle classification.

The weaker topology is plainly metrizable in the cases of interest here ($\bG$ compact Lie), as it identifies $E\bG$ with a subspace
\begin{equation*}
  E\bG
  =
  \left\{(t_ig_i)\in \cC\bG\ :\ \sum_i t_i=1,\ t_j=0\text{ for }j\gg 0\right\}
  \left(\subset \cC\bG\right)^{\bZ_{>0}}
\end{equation*}
where $\cC\bG:=\bG\times [0,1]/\bG\times \{0\}$ is the \emph{cone} \cite[p.9]{hatch_at} on $\bG$.

For the purpose of extending maps into $E\bG$ (or $B\bG$) defined on closed subspaces of paracompact spaces, it is relevant \cite[Theorem 3.1(b)]{MR59541} whether or not the metric topology on $E\bG$ is \emph{complete} (i.e. \cite[Problem 6.K]{kel-top} can be induced by a complete metric). Equivalently, $E\bG$ plainly being separable, the question is whether it is \emph{Polish} (separable + completely metrizable: \cite[Definition 3.1]{kech_descr}). That for compact Lie $\bG$ the metric topology on the Milnor total space $E\bG$ is \emph{not} Polish is not difficult to see: complete metrizability would entail \cite[Corollary 25.4(b)]{wil_top} $E\bG$'s being a \emph{Baire space} \cite[Definition 25.1]{wil_top}, which it cannot be having been exhibited as a countable union \Cref{eq:egeng} of nowhere-dense closed subsets.

\begin{remark}\label{re:not.top.compl}
  One issue with the finer topology \Cref{item:stop} above, hinted at in \cite[footnote 1]{atiyahsegal}, is the continuity of the action $E\bG\times \bG\to E\bG$: the endofunctor
  \begin{equation*}
    \left(
      \cat{Top}
      :=
      \text{Hausdorff topological spaces}
    \right)
    \xrightarrow{\quad -\times \bG\quad}
    \cat{Top}
  \end{equation*}
  will not, typically, preserve colimits. It does precisely \cite[Proposition 2.16]{MR4811433} when $\bG$ is \emph{exponentiable} in the sense of \cite[Definition 7.1.3]{brcx_hndbk-2}. As we are restricting attention to Hausdorff topological groups, being exponentiable is equivalent \cite[Theorems II-4.12 and V-5.6]{ghklms_cont-latt_2003} to being locally compact (though locally compact spaces are exponentiable even absent any separation assumptions: \cite[Proposition 7.1.5]{brcx_hndbk-2}). 
  
  An alternative to imposing local compactness or other such constraints is to work in the category $\cat{Top}_{\kappa}$ of (Hausdorff) \emph{compactly-generated} or \emph{$\kappa$-spaces} (\cite[Definition 7.2.5]{brcx_hndbk-2}, initially introduced in \cite[\S 1]{MR210075}). Said category is \emph{Cartesian closed} by \cite[Corollary 7.2.6]{brcx_hndbk-2}, i.e. all endofunctors of the form
  \begin{equation*}
    -\times_{\kappa} \bG
    ,\quad
    \otimes_{\kappa}:=\text{product in }\cat{Top}_{\kappa}
  \end{equation*}
  are left adjoints.

  This is the approach adopted (sometimes tacitly) in sources where colimit topologies such as \Cref{item:stop} and close analogues on classifying spaces are favored: \cite[\S 1, p.105]{zbMATH03317644}, \cite[\S 0]{MR353298}, \cite[\S 16.5]{may_at_1999}, \cite[p.1]{may_cls}, \cite[\S 1.1]{MR3704236}, and so forth.
  
  One exception to this last remark appears to be \cite[Definition 7.2.7]{hjjm_bdle}: the authors do not seem to me to be working in the category of $\kappa$-spaces, nor are there any constraints explicitly imposed on the topology of the group $\bG$. This is not an idle concern: for ill-behaved $\bG$ the translation action on the colimit-topologized $E\bG$ can be discontinuous, per \Cref{pr:q.discont.act.join}. 
\end{remark}

\begin{proposition}\label{pr:q.discont.act.join}
  The free diagonal right action of $\left(\bQ,+\right)$ on $\bQ*\bQ$ is discontinuous if the join is equipped with the quotient topology. 
\end{proposition}
\begin{proof}
  We progressively simplify the space acted upon, bringing it into the scope of familiar examples of quotient-topology misbehavior under taking products.

  Note in first instance that it will be enough to prove the discontinuity of the action on the ``truncated'' join
  \begin{equation}\label{eq:trunc.join}
    \begin{aligned}
      \bQ\times \bQ\times \left[\frac 12,1\right]
      \bigg/
      (q,q',1)\sim (q,q'',1)
      &\cong
        \left\{tq+(1-t)q',\ t\in \left[\frac 12,1\right],\ q,q'\in \bQ\right\}\\
      &\subset
        \bQ*\bQ
        =
        \left\{tq+(1-t)q',\ t\in \left[0,1\right],\ q,q'\in \bQ\right\}
    \end{aligned}    
  \end{equation}
  equipped with its quotient topology. The map
  \begin{equation*}
    \text{\Cref{eq:trunc.join}}
    \ni
    tq+(1-t)q'
    \xmapsto{\quad}
    \left(q,\frac{1-t}{t}\right)
    \in
    \bQ\times \cC \bQ
    =
    \bQ\times \left(\bQ\times [0,1]\bigg/(q,0)\sim (q',0)\right)
  \end{equation*}
  is bijective, and a homeomorphism if we equip the Cartesian product not with its product topology, but rather with the quotient topology resulting from
  \begin{equation*}
    \bQ\times \bQ\times [0,1]
    \xrightarrowdbl{\quad}
    \bQ\times \cC \bQ;
  \end{equation*}
  we write $\bQ\;\wt{\times}\; \cC \bQ$ for that finer, quotient topology on the product. A further (self-)homeomorphism
  \begin{equation*}
    \bQ\;\wt{\times}\; \cC \bQ
    \ni
    \left(q,tq'\right)
    \xmapsto[\quad\cong\quad]{\quad}
    \left(q,t\left(q'-q\right)\right)
    \in
    \bQ\;\wt{\times}\; \cC \bQ
  \end{equation*}
  intertwines the diagonal $\bQ$-action on the domain and the left-hand-factor translation on the codomain, so the claim has been reduced to showing that
  \begin{equation}\label{eq:qqq}
    \bQ\times \left(\bQ\;\wt{\times}\; \cC \bQ\right)
    \xrightarrow{\quad (+) \times \id_{\cC\bQ}\quad}
    \bQ\;\wt{\times}\; \cC \bQ
  \end{equation}
  is discontinuous. To conclude, observe that (slightly paraphrased) \cite[Example 1.5.11]{dvr-bk} confirms even the stronger claim that
  \begin{equation}\label{eq:qqz}
    \bQ\times \left(\bQ\;\wt{\times}\; \cC \bZ_{\ge 0}\right)
    \xrightarrow{\quad (+) \times \id_{\cC\bQ}\quad}
    \bQ\;\wt{\times}\; \cC \bZ_{\ge 0}
  \end{equation}
  is discontinuous (stronger because quotient maps are compatible with passing to closed subsets \cite[Theorem 22.1(1)]{mnk}; a continuous action \Cref{eq:qqq} would thus restrict to one of the form \Cref{eq:qqz}).
\end{proof}

\begin{remarks}\label{res:q.on.cone}
  \begin{enumerate}[(1),wide]
  \item\label{item:res:q.on.cone:act.cone} In reference to the discontinuity of both the diagonal and left-hand-translation actions
    \begin{equation*}
      \bQ
      \times
      \left(
        \bQ\;\wt{\times}\; \cC \bQ
      \right)
      \xrightarrow{\quad}
      \bQ\;\wt{\times}\; \cC \bQ,
    \end{equation*}
    note that this holds also of the obvious action
    \begin{equation}\label{eq:qcq}
      \bQ\times \cC \bQ
      \xrightarrow{\quad\triangleright\quad}
      \cC \bQ
    \end{equation}
    on the quotient-topologized cone. Indeed:
    \begin{itemize}[wide]
    \item Consider an increasing sequence
      \begin{equation*}
        \bR\setminus \bQ
        \ni
        r_n
        \xrightarrow[\quad n\to \pm\infty\quad]{\quad}
        \pm \infty
        ,\quad
        r_{n+1}-r_n
        \xrightarrow[\quad n\to\infty \quad]{\quad}
        0;
      \end{equation*}

    \item Let
      \begin{equation*}
        \bQ
        \xrightarrow{\quad\varphi\quad}
        [0,1]
      \end{equation*}
      be the restriction of a continuous function on $\bR$ vanishing precisely at the $r_n$;

    \item Let $U\subset \cC \bQ$ be the neighborhood of $\bQ\times \{0\}$ in $\bQ\times [0,1]$ (which thus surjects onto a neighborhood of the cone tip) defined as the region
      \begin{equation*}
        \left\{(q,t)\ :\ t<\varphi(q)\right\}
        \subset
        \bQ\times [0,1]
      \end{equation*}
      under the graph of $\varphi$;

    \item It is now an easy check that for no open neighborhoods
      \begin{equation*}
        V\ni \left(\text{any point $p$}\right)\in \bQ\times \{0\}\subset \bQ\times [0,1]
        \quad\text{and}\quad
        W\ni 0\in \bQ
      \end{equation*}
      do we have $W\triangleright V\subset U$.
    \end{itemize}

  \item\label{item:res:q.on.cone:act.cone.enough} Item \Cref{item:res:q.on.cone:act.cone} above in fact provides an alternative proof for \Cref{pr:q.discont.act.join}: were \Cref{eq:qqq} continuous, it would descend to a continuous \Cref{eq:qcq} by collapsing the $\bQ$ factor.

  \item\label{item:res:q.on.cone:lcpct.nnec} While local compactness is certainly \emph{sufficient} (per \Cref{re:not.top.compl}) to ensure that the action of $\bG$ on $\left(E\bG,\tau_{\varinjlim}\right)$ and hence also the individual $\left(E_n\bG,\tau_{\varinjlim}\right)$ is continuous, \Cref{ex:sspace} below shows that it is not \emph{necessary} for the latter to hold.
  \end{enumerate}  
\end{remarks}

We record the more general result that \Cref{res:q.on.cone}\Cref{item:res:q.on.cone:act.cone.enough} and the proof of \Cref{pr:q.discont.act.join} and in fact settle.

\begin{theorem}\label{th:disc.cone}
  If a Hausdorff topological group $\bG$
  \begin{itemize}[wide]
  \item acts discontinuously on its cone $\cC \bG$ with the quotient topology;

  \item or has distinct quotient and product topologies on $\bG\times \cC \bG$
  \end{itemize}
  then it acts discontinuously on its quotient-topologized join $\bG*\bG$, and hence also on $E_n \bG$, $1\le n\le \infty$.  \qedhere
\end{theorem}

\begin{example}\label{ex:sspace}
  This construction is carried out in \cite[Discussion, p.485]{zbMATH03340311} (also \cite[p.38]{MR133392}): the subgroup
  \begin{align*}
    \bG
    &:=
      \left\{(g_i)_i\in \prod_i \bG_i\ :\ \text{at most countably many }g_i\ne 1\right\}\\
    &\le
      \ol{\bG}
      :=
      \prod_{i\in I}\bG_i
      ,\quad
      \bG_i\text{ compact metrizable groups}.
  \end{align*}
  For uncountable index sets $I$ the group $\bG$ will be non-compact but \emph{countably compact} in the sense \cite[p.19]{ss_countertop} that countable open covers have finite subcovers. The latter property (valid for all powers $\bG^{n}$) will ensure that the subspace and quotient topologies agree on $E_n\bG\subseteq E_n\ol{\bG}$. Indeed, in both cases (the countably-compact $\bG$ as well as the compact $\ol{\bG}$) the quotient and coordinate topologies contrasted in \Cref{item:stop} and \Cref{item:wtop} above agree by \Cref{le:colim.top.eq.coord.top} below. The continuous action
  \begin{equation*}
    E_n\ol{\bG}\times \ol{\bG}
    \xrightarrow{\quad}
    E_n\ol{\bG}
  \end{equation*}
  thus restricts continuously to its $\bG$ analogue, concluding the argument. 
\end{example}

\begin{lemma}\label{le:colim.top.eq.coord.top}
  Let $n\in \bZ_{\ge 0}$ and consider Hausdorff spaces $\left\{X_i\right\}_{i=0}^n$. If the product of any $n$ $X_i$ is countably compact then the colimit and coordinate topologies on the join $X_0*\cdots * X_n$ coincide. 
\end{lemma}
\begin{proof}
  We can proceed by induction on $n$, writing
  \begin{equation*}
    X_0*\cdots * X_n
    \cong
    X*Y
    ,\quad
    \begin{aligned}
      X&:=X_0\\
      Y&:=X_1*\cdots*X_n
    \end{aligned}
  \end{equation*}
  with both $X$ and $Y$ countably-compact: 
  \begin{itemize}[wide]
  \item the colimit and coordinate topologies coincide on $Y=X_1*\cdots*X_n$ by the induction hypothesis;

  \item and that single topology on the quotient
    \begin{equation*}
      X_1\times \cdots\times X_n\times [0,1]^{n-1}
      \xrightarrowdbl{\quad}
      Y
    \end{equation*}
    is countably compact, that property being inherited by quotient topologies as well as products with compact Hausdorff spaces \cite[Theorem 5]{MR60212}. 
  \end{itemize}
  It thus suffices to prove the claim for joins $X*Y$ of countably-compact spaces $X,Y$.

  The coordinate function $t$ is in any case continuous (for both topologies) on
  \begin{equation*}
    X*Y = \left\{tx+(1-t)y\ :\ (x,y,t)\in X\times Y\times [0,1]\right\},
  \end{equation*}
  so it suffices to examine either one of the analogous closed subspaces
  \begin{equation*}
    X*Y_{t\le \frac 23}
    \quad\text{and}\quad
    X*Y_{t\ge \frac 13}
    \quad
    \subseteq
    \quad
    X*Y,
  \end{equation*}
  whose interiors cover $X*Y$. Focusing on the latter, we have homeomorphisms
  \begin{equation*}
    X*Y_{t\ge \frac 13}
    \ni
    tx+(1-t)y
    \xmapsto[\quad\cong\quad]{\quad}
    \left(x,\frac{1-t}{t}y\right)
    \in
    X\times \cC Y
  \end{equation*}
  identifying the target space with the product in either the quotient topology induced by $X\times Y\times [0,1]\xrightarrowdbl{\hspace{0pt}}X\times \cC Y$ or the product topology. That these two agree for countably-compact $Y$ follows from a combination of \cite[Lemmas 3 and 4]{MR1992473} in much the same way as the analogous result for plain $\cC Y$ does. 
\end{proof}

\begin{remark}
  In somewhat different language, \cite[Lemmas 3 and 4]{MR1992473} provide a version of \Cref{le:colim.top.eq.coord.top} for cones: for countably-compact $X$ the coordinate and quotient topologies on $\cC X$ coincide. 
\end{remark}

We now record the more general positive result the argument driving \Cref{ex:sspace} helps prove. 

\begin{theorem}\label{th:lb.cc}
  For Hausdorff topological groups $\bG$ with $\bG^n$ countably-compact the $\bG$-action on the universal space $\left(E_n\bG,\tau_{\varinjlim}\right)$ is continuous.
\end{theorem}
\begin{proof}  
  Recall \cite[p.20]{ss_countertop} that \emph{pseudocompact} spaces are those whose real-valued continuous functions are bounded, and note that (local) countable compactness (of $\bG$ alone) entails (respectively local) pseudocompactness, which in turn implies local \emph{boundedness} \cite[p.267]{MR1326826}: some identity neighborhood can be covered by finitely many translates of any given non-empty open set. The locally-bounded Hausdorff topological groups are precisely \cite[Theorems 1.6 and 1.7]{MR1326826} those homeomorphically embeddable into locally compact groups, hence such an embedding $\bG\le \ol{\bG}$.

  Now, $\ol{\bG}$ acts continuously on \emph{its} iterated joins $E_n\ol{\bG}$ by \Cref{re:not.top.compl} and \begin{equation*}
    \left(E_n\bG,\tau_{\varinjlim}\right)
    \lhook\joinrel\xrightarrow{\quad}
    \left(E_n\ol{\bG},\tau_{\varinjlim}\right)
  \end{equation*}
  is a homeomorphism onto its image by \Cref{le:colim.top.eq.coord.top}. The continuous $\ol{\bG}$-action restricts to that of $\bG$, hence the conclusion.
\end{proof}

\begin{remark}\label{re:pcpct.not.enough}
  Having mentioned pseudocompactness in the course of the proof of \Cref{th:lb.cc}, it is perhaps worth noting that that specific argument would \emph{not} have functioned assuming that property alone: so long as the Hausdorff space $X$ is not countably compact, a countable open cover $X=\bigcup_n U_n$ with no finite subcover would produce a neighborhood
  \begin{equation*}
    \bigcup_n \left(U_n\times \left[0,\frac 1n\right)\right)
    \supseteq
    X\cong X\times \{0\}
    \subseteq
    X\times [0,1]
  \end{equation*}
  descending to one of the cone tip in the quotient but not the coordinate topology.

  Pseudocompact, non-countably-compact groups do exist \cite[pp.38-39]{MR133392}: the subgroup $\bG\le \ol{\bG}:=\left(\bS^1\right)^{\aleph_1}$ consisting of tuples with at most countably many non-root-of-unity components. 
\end{remark}



\addcontentsline{toc}{section}{References}

\def\polhk#1{\setbox0=\hbox{#1}{\ooalign{\hidewidth
  \lower1.5ex\hbox{`}\hidewidth\crcr\unhbox0}}}
  \def\polhk#1{\setbox0=\hbox{#1}{\ooalign{\hidewidth
  \lower1.5ex\hbox{`}\hidewidth\crcr\unhbox0}}}
  \def\polhk#1{\setbox0=\hbox{#1}{\ooalign{\hidewidth
  \lower1.5ex\hbox{`}\hidewidth\crcr\unhbox0}}}
  \def\polhk#1{\setbox0=\hbox{#1}{\ooalign{\hidewidth
  \lower1.5ex\hbox{`}\hidewidth\crcr\unhbox0}}}
  \def\polhk#1{\setbox0=\hbox{#1}{\ooalign{\hidewidth
  \lower1.5ex\hbox{`}\hidewidth\crcr\unhbox0}}}
  \def\polhk#1{\setbox0=\hbox{#1}{\ooalign{\hidewidth
  \lower1.5ex\hbox{`}\hidewidth\crcr\unhbox0}}}

\Addresses

\end{document}